\documentclass[12pt]{amsart}

\usepackage{amssymb}
\usepackage{bm}
\usepackage{graphicx}
\usepackage{psfrag}
\usepackage[usenames,dvipsnames,table]{xcolor}
\usepackage{enumerate}
\usepackage{multirow}
\usepackage{url}
\usepackage[labelformat=simple]{subcaption}
\usepackage{comment}

\usepackage{algpseudocode}

\usepackage{mathtools}

\usepackage{xy}
\input xy
\xyoption{all}

\newcommand{\excise}[1]{}%{$\star$\textsc{#1}$\star$}
\newtheorem{thm}{Theorem}[section]
\newtheorem{lemma}[thm]{Lemma}

\newtheorem{cor}[thm]{Corollary}
\newtheorem{prop}[thm]{Proposition}

\newtheorem{question}[thm]{Question}

\newtheorem*{mainthm}{Structure Theorem for Sets of Lengths}

\theoremstyle{definition}

\newtheorem{example}[thm]{Example}
\newtheorem{remark}[thm]{Remark}
\newtheorem{defn}[thm]{Definition}

\newtheorem{notation}[thm]{Notation}

\numberwithin{equation}{section}

\renewcommand\>{\rangle}
\newcommand\<{\langle}

\newcommand\ZZ{\mathbb{Z}}

\DeclareMathOperator\lcm{lcm} % lcm
\newcommand{\Frob}{\mathsf{Frob}}

\newcommand{\Ap}{\mathsf{Ap}}

\newcommand{\MinS}{\mathsf{m}}
\newcommand{\M}{\mathsf{M}}
\newcommand{\LengthSet}{\mathsf{L}}

\definecolor{lightred}{rgb}{1.0, 0.7, 0.7}

\begin{document}

\mbox{}
%\vspace{-2ex}%-1.1743pt}
\title[The structure theorem for sets of length for numerical semigroups]{The structure theorem for sets of length \\ for numerical semigroups}

\author[Moskowitz]{Gilad Moskowitz}
\address{Mathematics Department\\San Diego State University\\San Diego, CA 92182}
\email{gilad.moskowitz@gmail.com}

\author[O'Neill]{Christopher O'Neill}
\address{Mathematics Department\\San Diego State University\\San Diego, CA 92182}
\email{cdoneill@sdsu.edu}

\begin{abstract}
For sufficiently nice families of semigroups and monoids, the structure theorem for sets of length states that the length set of any sufficiently large element is an arithmetic sequence with some values omitted near the ends.  In this paper, we prove a specialized version of the structure theorem that holds for any numerical semigroup $S$.  Our description utilizes two other numerical semigroups $S_{\mathsf M}$ and $S_{\mathsf m}$, derived from the generators of $S$:\ for sufficiently large $n \in S$, the Ap\'ery sets of $S_{\mathsf M}$ and $S_{\mathsf m}$ specify precisely which lengths appear in the length set of $n$, and their gaps specify which lengths are ``missing''.  We also provide an explicit bound on which elements satisfy the structure theorem.  
\end{abstract}

\maketitle

%%%%%%%%%%%%%%%%%%%%%%%%%%%%%%%%%%%%%%%%%%%%%%%%%%%%%%%%%%%%%%%%%%%%%%%%%
\section{Introduction}%%%%%%%%%%%%%%%%%%%%%%%%%%%%%%%%%%%%%%%%%%%%%%%%%%%
\label{sec:intro}%%%%%%%%%%%%%%%%%%%%%%%%%%%%%%%%%%%%%%%%%%%%%%%%%%%%%%%%
%raggedbottom%%%%%%%%%%%%%%%%%%%%%%%%%%%%%%%%%%%%%%%%%%%%%%%%%%%%%%%%%%%%

Throughout this document, we let $S$ denote a \emph{numerical semigroup}  (that is, an additively closed subset of $\ZZ_{\ge 0}$), and denote by $n_1, \ldots, n_k$ a generating set of $S$, i.e.,
$$S = \<n_1, \ldots, n_k\> = \{q_1n_1 + q_2n_2 + \cdots + q_kn_k \mid q_1, \ldots, q_k \in \ZZ_{\ge 0}\}.$$
It is known that a numerical semigroup $S$ is cofinite in $\ZZ_{\ge 0}$ if and only if \(\gcd(S) = 1\), and it is common practice to assume this holds.  It is also common practice to assume $n_1, \ldots, n_k$ comprise the unique minimal generating set of $S$.  
However, in this paper, we do \textbf{not} make either of these assumptions.  

A \emph{factorization} of $n \in S$ is an expression 
\[
n = q_1n_1 + \cdots + q_kn_k
\]
of $n$ as a sum of generators of $S$, and the \emph{length} of a factorization is the sum $q_1 + \cdots + q_k$.  
% The \emph{set of factorizations} of $n$ is the set
% \[
% \mathsf Z_S(n) = \{q \in \ZZ_{\ge 0}^k : n = q_1n_1 + \cdots + q_kn_k\}
% \]
% viewed as a subset of $\ZZ_{\ge 0}^k$, and 
The \emph{length set} of $n$ is the set 
$$\mathsf L_S(n) = \{q_1 + \cdots + q_k : q_1, \ldots, q_k \in \ZZ_{\ge 0} \text{ with } n = q_1n_1 + \cdots + q_kn_k\}$$
of all possible factorization lengths of $n$.  Define
% The \emph{maximum} and \emph{minimum} factorization length functions are defined as 
$$\mathsf M_S(n) = \max \mathsf L_S(n) \qquad \text{ and } \qquad \mathsf m_S(n) = \min \mathsf L_S(n).$$
When there can be no confusion, we often omit the subscripts and simply write $\mathsf L(n)$, $\mathsf M(n)$, and $\mathsf m(n)$, respectively.  

% For a thorough introduction to numerical semigroups, see~\cite{numerical}.  

The structure theorem for sets of length~\cite{geroldingerlengthsets}, a cornerstone of factorization theory, states that for any sufficiently large semigroup element \(n\), the length set of \(n\) will be an almost arithmetical progression (that is, an arithmetic sequence 
% starting with \(\MinS(n)\), ending with \(\M(n)\), with a common difference \(d\), and 
with a few elements missing towards the beginning and end of the sequence).  The scope of the structure theorem goes well beyond that of numerical semigroups; it is known to hold for a broad family of semigroups and monoids, including finitely presented monoids, large families of Krull monoids, and others; see the monograph~\cite{nonuniq} for a thorough overview.  In fact, one of the central themes in factorization theory is determining for which families of semigroups the structure theorem holds; see~\cite{setsoflengthmonthly} for a detailed account.  

We now state the structure theorem in the current context of numerical semigroups.  

\begin{mainthm}
There exist integers \(t, t'\) and \(d\) such that for sufficiently large \(n \in S\), there exist \(A \subseteq [1, t]\) and \(A' \subseteq [1, t']\) with the property that
\[\LengthSet(n) = \{\MinS(n), \MinS(n) + d, \dots, \M(n) - d, \M(n)\} \setminus \big( (dA' + \MinS(n)) \cup (-dA + \M(n)) \big).\]  
\end{mainthm}

In recent years, there has been an effort to specialize the structure theorem for semigroups of sufficiently high interest, stemming in part from its connections to some of the biggest open problems in additive combinatorics~\cite{factorizationadditivebridge,lensetprogress}.  These specializations generally concern which length sets are possible~\cite{setoflengthsets,lensetrealiz,realizsystemlengthskrull,realizthm}, while others focus on refinements of the structure theorem, such as the unions of all sets of length~\cite{additivefamiliesfactorization}, or a description of the ``missing lengths'', both locally for elements~\cite{subdeltas,setdistanceskrull} and globally for the semigroup as a whole~\cite{deltarealizationnumerical,deltarsetealization}.  

The main result of the present paper is Theorem~\ref{t:structurethm}, a refined structure theorem for sets of length for numerical semigroups, wherein we characterize the values $d, t$ and $t'$ in the theorem, and identify bijections between the sets $A \subseteq [1,t]$ and $A' \subseteq [1,t']$ of missing factorization lengths and sets of gaps in the semigroups 
\[
S_{\M} =  \langle n_2 - n_1, n_3 - n_1, \dots, n_k - n_1 \rangle
\quad \text{and} \quad
S_{\MinS} =  \langle n_k - n_1, n_k - n_2, \dots, n_k - n_{k-1} \rangle
\]
respectively.  This is best illustrated with an example.  

\begin{example}\label{e:mainex}
Let $S = \langle 5,9,12 \rangle$.  Figure~\ref{f:maxlenset} depicts the ``top'' of the length sets $\mathsf L(100)$, $\ldots$, $\mathsf L(104)$, with filled black boxes indicating the ``missing'' lengths (the ``$A$'' sets in the structure theorem).  Figure~\ref{f:maxlenapery} depicts the elements of the semigroup 
$$S_{\M} = \langle 9 - 5, 12 - 5 \rangle = \langle 4, 7 \rangle$$
with filled black boxes indicating the gap set $\ZZ_{\ge 0} \setminus S_{\M}$.  Notice the identical positioning of the filled black boxes in each depiction; this relationship is the heart of Theorem~\ref{t:structurethm}.  Figures~\ref{f:minlenapery} and~\ref{f:minlenset} depicts a similar phenomenon (after a reflection) for the sets $A'$ in the structure theorem for the numerical semigroup $S' = \<4,6,9\>$.  
\end{example}

\begin{figure}[t]
\begin{subfigure}{.4\textwidth}
\begin{center}
\centering
\begin{tabular}{ | @{\,} c @{\,} | @{\,} c @{\,} | @{\,} c @{\,} | @{\,} c @{\,} | @{\,} c @{\,} |}
  \hline
  \textbf{\phantom{1}0\phantom{0}} & \textbf{\phantom{1}1\phantom{0}} & \textbf{\phantom{1}2\phantom{0}} & \textbf{\phantom{1}3\phantom{0}} & \textbf{\phantom{1}4\phantom{0}}\\
  \hline
  \hline
     0 & \cellcolor{black} & \cellcolor{black} & \cellcolor{black} & 4 \\
    \hline
    \cellcolor{black} & \cellcolor{black} & 7 & 8 & \cellcolor{black} \\
    \hline
    \cellcolor{black} & 11 & \underline{12} & \cellcolor{black} & \underline{14} \\
    \hline
    \underline{15} & \underline{16} & \cellcolor{black} & \underline{18} & 19 \\
    \hline
    20 & 21 & 22 & 23 & 24 \\
  \hline
\end{tabular}
\end{center}
\caption{Elements of \(S_{\M} = \langle 4, 7 \rangle\) below 25, arranged by equivalence class mod 5.}
\label{f:maxlenapery}
\end{subfigure}%
\qquad\qquad 
\begin{subfigure}{.2\textwidth}
\end{subfigure}%
\begin{subfigure}{.4\textwidth}
\begin{center}
\centering
\begin{tabular}{ | @{\,} c @{\,} | @{\,} c @{\,} | @{\,} c @{\,} | @{\,} c @{\,} | @{\,} c @{\,} |}
  \hline
  \textbf{100} & \textbf{101} & \textbf{102} & \textbf{103} & \textbf{104}\\
  \hline
  \hline
    20 & \cellcolor{black} & \cellcolor{black} & \cellcolor{black} & 20 \\
    \hline
    \cellcolor{black} & \cellcolor{black} & 19 & 19 & \cellcolor{black} \\
    \hline
    \cellcolor{black} & 18 & 18 & \cellcolor{black} & 18 \\
    \hline
    17 & 17 & \cellcolor{black} & 17 & 17 \\
    \hline
    16 & 16 & 16 & 16 & 16 \\
  \hline
\end{tabular}
\end{center}
\caption{Tops of the length sets of the elements \(100,\ldots,104 \in S = \langle 5, 9, 12 \rangle\).}
\label{f:maxlenset}
\end{subfigure}

\vspace{1em}

\begin{subfigure}{.4\textwidth}
\begin{center}
\centering
\begin{tabular}{ | @{\,\,} c @{\,\,} | @{\,} c @{\,} | @{\,} c @{\,} | @{\,} c @{\,} | @{\,} c @{\,} | @{\,} c @{\,} | @{\,} c @{\,} | @{\,} c @{\,} | @{\,} c @{\,} | }
  \hline
  \textbf{0} & \textbf{1} & \textbf{2} & \textbf{3} & \textbf{4} & \textbf{5} & \textbf{6} & \textbf{7} & \textbf{8}\\
  \hline
  \hline
    0 & \cellcolor{black} & \cellcolor{black} & 3 & \cellcolor{black} & 5 & 6 & \cellcolor{black} & 8\\
    \hline
    9 & 10 & 11 & 12 & 13 & 14 & 15 & 16 & 17\\
    \hline
\end{tabular}
\end{center}
\caption{Elements of \(S_{\MinS}' = \langle 3, 5 \rangle\) below 18, arranged by equivalence class mod 9.}
\label{f:minlenapery}
\end{subfigure}%
\qquad\qquad 
\begin{subfigure}{.2\textwidth}
\end{subfigure}%
\begin{subfigure}{.4\textwidth}
\begin{center}
\centering
\begin{tabular}{ | @{\,} c @{\,} | @{\,} c @{\,} | @{\,} c @{\,} | @{\,} c @{\,} | @{\,} c @{\,} | @{\,} c @{\,} | @{\,} c @{\,} | @{\,} c @{\,} | @{\,} c @{\,} |}
  \hline
  \textbf{91} & \textbf{92} & \textbf{93} & \textbf{94} & \textbf{95} & \textbf{96} & \textbf{97} & \textbf{98} & \textbf{99}\\
  \hline
  \hline
    11 & \cellcolor{black} & 11 & 11 & \cellcolor{black} & 11 & \cellcolor{black} & \cellcolor{black} & 11\\
    \hline
    12 & 12 & 12 & 12 & 12 & 12 & 12 & 12 & 12\\
    \hline
\end{tabular}
\end{center}
\caption{Bottoms of the length sets of the elements \(91,\ldots,99 \in S' = \langle 4, 6, 9 \rangle\).}
\label{f:minlenset}
\end{subfigure}
\caption{}
\label{f:maxminlentables}
\end{figure}

Our result comes as part of a recently flurry of papers examining the factorization properties of large numerical semigroup elements, many of which turn out to be eventually periodic or quasipolynomial~\cite{elastsets,delta,catenaryperiodic,omegaquasi}; see the survey~\cite{numericalsurvey} for details and~\cite{compasympomega,compasympdelta} for computational applications.  
The primary strength of our result is that it characterizes the missing lengths in terms of gap sets~\cite{gapsnonsym,fundamentalgaps,diophantinefrob}, which have been a central focus in the study of numerical semigroups since their inception~\cite{frobsylvester}.  

The paper is organized as follows.  After introducing a generalization of the Ap\'ery set in Section~\ref{sec:generalapery}, we prove in Section~\ref{sec:maxminlen} that for sufficiently large $n$, the set $A$ from the structure theorem is identical for $\mathsf L(n)$ and $\mathsf L(n + n_1)$, and the set $A'$ is identical for $\mathsf L(n)$ and $\mathsf L(n + n_k)$ (Theorems~\ref{t:jmaxlen} and~\ref{t:jminlen}, respectively).  In Section~\ref{sec:structurethm}, we prove Theorem~\ref{t:structurethm}, characterizing the sets $A$ and $A'$ in terms of the gaps of the semigroups $S_{\M}$ and $S_{\MinS}$, respectively, as well as obtain an explicit bound on the $n \in S$ for which the structure theorem holds (Theorem~\ref{t:bound}).   We also draw conclusions about realization questions akin to those considered in~\cite{subdeltas,deltarealizationnumerical,deltarsetealization,setdistanceskrull} for other families of semigroups and monoids; see the discsussion in Remark~\ref{r:classification}.

%%%%%%%%%%%%%%%%%%%%%%%%%%%%%%%%%%%%%%%%%%%%%%%%%%%%%%%%%%%%%%%%%%%%%%%%%
\section{A generalization of the Ap\'ery set}%%%%%%%%%%%%%%%%%%%%%%%%%%%%
\label{sec:generalapery}%%%%%%%%%%%%%%%%%%%%%%%%%%%%%%%%%%%%%%%%%%%%%%%%%
%raggedbottom%%%%%%%%%%%%%%%%%%%%%%%%%%%%%%%%%%%%%%%%%%%%%%%%%%%%%%%%%%%%

The Ap\'ery set of a numerical semigroup $T$ is central to both theoretical~\cite{aperyhilbert} and computational~\cite{compapery} aspects of numerical semigroups; see~\cite{numericalappl} for a thorough overview.  Usually defined with respect to an element $n \in T$, the Ap\'ery set 
\[
\Ap(T;n) = \{m \in T : m - n \notin T\}
\]
can be shown to consist of the first element of $T$ in each equivalence class modulo $n$.  
In this section, we define a generalization of the Ap\'ery set that allows $n \in \ZZ_{\ge 0}$ and $\gcd(T) > 1$.  Other generalizations of the Ap\'ery set have been studied, and while some are similar to our definition~\cite{valuesetplanebranches}, most allow $\Ap(T;n)$ to contain more than one element of each equivalence class modulo $n$ if $n \notin T$~\cite{heirarchycodesns,telescopiccodes,goppalikefengrao,hwintersection}.  Moreover, none that the authors were able to find allowed $\gcd(T) > 1$.  
After verifying some basic properties of $\Ap(T;n)$, we introduce a collection of sets that partition $T$, with $\Ap(T;n)$ as its foundation, that will play a key role in subsequent sections.  

\begin{notation}\label{n:generalapery}
Throughout this section, let $S \subseteq \ZZ_{\ge 0}$ denote a cofinite numerical semigroup, let \(d \in \ZZ_{\ge 1}\), and $T = dS \subseteq \ZZ_{\ge 0}$.
\end{notation}

\begin{defn}\label{d:aperyset}
Fix \(n \in \ZZ_{\ge 1}\). For each $i \in \{0, 1, \ldots, n-1\}$, let
\[
a_i = \begin{cases}
0 & \text{if \(T \cap \{i, i+n, i+2n,\ldots\} = \emptyset\);} \\ 
\min(T \cap \{i, i+n, i+2n,\ldots\}) & \text{otherwise.} \end{cases}
\]
The \emph{Ap\'ery set} of \(T\) with respect to \(n\) as
\[\Ap(T; n) = \{a_i \mid i = 0, 1, \ldots, n-1\}.\]
Note that if \(T\) has finite complement and \(n \in T\), then \(\Ap(T; n)\) coincides with the usual definition of the Ap\'ery set~\cite{numericalappl}.  
\end{defn}

We briefly verify that under mild hypotheses, $\Ap(T;n)$ has some familiar properties.  

\begin{prop}\label{p:aperyset}
For any \(n \in \ZZ_{\ge 1}\), the elements of \(\Ap(T; n)\) are distinct modulo \(n\).  Moreover, if \(\gcd(d, n) = 1\), then \(|\Ap(T; n)| = n\) and \(\Ap(T; dn) = \Ap(T; n)\).  
\end{prop}

\begin{proof}
The first claim follows from the definition of $\Ap(T; n)$ since $a_0 = 0$ and each nonzero $a_i$ satisfies $a_i \equiv i \bmod n$.  Next, fixing $z \in \ZZ$ so that $z + d \ZZ_{\ge 0} \subseteq T$, we see \(y = nz + d\) satisfies $y \in T$ and $y \equiv d \bmod n$.  
This means if \(\gcd(d, n) = 1\), then taking integer multiples of $y$ reaches each equivalence class modulo $n$, so $\Ap(T; n)$ contains an element from each equivalence class modulo $n$, and thus \(|\Ap(T; n)| = n\).  
For the final claim, it suffices to observe \(\Ap(T; n) \subseteq \Ap(T; dn)\) (since $a - dn \equiv a - n \bmod n$ for each $a \in \Ap(T;n)$) and $|\Ap(T; dn)| \le n$ (since each $a \in \Ap(T; dn)$ satisfies $d \mid a$).  
\end{proof}

% \begin{remark}
% Note that the primary difference between Proposition~\ref{p:aperyset} \emph{a} and \emph{b} is that in \emph{a} we are looking at the Ap\'ery set of \(T\) with respect to an integer that is relatively prime to \(\gcd(T)\), and in \emph{b} we look at the Ap\'ery set with respect to an integer that is a multiple of \(\gcd(T)\).
% \end{remark}

\begin{defn}\label{d:japeryset}
Fix \(n \in \ZZ_{\ge 1}\) and \(j \geq 1\).  The \emph{j-th Ap\'ery set} of \(T\) with respect to \(n\) is the set $\Ap_j(T; n)$ consisting of the $j$-th element of $\{a, a + n, \ldots\} \cap T$  for each $a \in \Ap(T; n)$.  In particular, 
\[
\Ap_j(T; n) = \{a + kn \in T : a \in \Ap(T; n) \text{ and } |\{a, a + n, \ldots, a + kn\} \cap T| = j\},
\]
where for each $a \in \Ap(T; n)$, there is a unique $k \in \ZZ_{\ge 0}$ such that $a + kn \in \Ap_j(T; n)$.  
\end{defn}

\begin{example}\label{e:aperyset}
Let \(T = \langle 4, 7 \rangle\), whose first few elements are
\[T = \{0, 4, 7, 8, 11, 12, 14, 15, 16, 18, 19, 20, \ldots\}.\]
Under Definition~\ref{d:aperyset}, we have
\[\Ap(T; 5) = \{0, 11, 7, 8, 4\},\]
and under Definition~\ref{d:japeryset}, we have $\Ap_1(T; 5) = \Ap(T; 5)$,
\[
\Ap_2(T; 5) = \{15, 16, 12, 18, 14\},
\quad \text{and} \quad
\Ap_3(T; 5) = \{20, 21, 22, 23, 19\}.
\]
These comprise the first, second, and third integers in each column in Figure~\ref{f:maxlenapery}, respectively.  
For each \(j \geq 3\), we see that \(\Ap_{j+1}(T; 5) = \Ap_{j}(T; 5) + 5\) since \(\ZZ_{\ge 19} \subseteq T\) and \(\min \Ap_3(T; 5) = 19\).
\end{example}

\begin{lemma}\label{l:japeryset}
% For all $n \in \ZZ_{\ge 1}$ and $j \in \ZZ_{\ge 1}$, 
% Each nonzero element in \(\Ap_j(T; n)\) is the \(j\)-th element in its equivalence class modulo \(n\) that lies in \(T\).  In particular, 
For each $n \in \ZZ_{\ge 1}$, \(\Ap_1(T; n) = \Ap(T; n)\), and
\[
T = \bigcup_{j \ge 1} \Ap_j(T; n),
\]
where the right hand side in the above equality is a disjoint union.  
If \(n \in T\), then 
\[\Ap_j(T; n) = \Ap(T; n) + (j-1)n = \{a + (j-1)n : a \in \Ap(T; n) \}.\]
\end{lemma}

\begin{proof}
All claims follow from induction on $j$ and the fact that for each $a \in \Ap_j(T;n)$, choosing $k \in \ZZ_{\ge 1}$ minimal so that $a + kn \in T$ ensures $a + kn \in \Ap_{j+1}(T;n)$.  
\end{proof}

We close this section with one final definition and lemma we will use in Section~\ref{sec:structurethm}.  

\begin{defn}\label{d:frob}
The \emph{Frobenius number} of \(T\) is
\[\Frob(T) = d (\max\Ap(S; n_1) - n_1).\]
where \(n_1\) is the smallest generator of \(S\).
When \(d = 1\), we obtain
\[\Frob(S) = \max\big(\Ap(S; n_1)\big) - n_1\]
which coincides with the traditional definition of the Frobenius number.
\end{defn}

\begin{lemma}\label{l:minlenupperbound}
Suppose $T = \<n_1, \ldots, n_k\>$.  If $n \in T$, then
\[
\tfrac{1}{n_k}n \le \MinS(n) \le \tfrac{1}{n_k}n + (n_k - n_1)
\qquad \text{and} \qquad
\M(n) \le \tfrac{1}{n_1}n.
\]
\end{lemma}

\begin{proof}
The first and last inequalities above follow from the fact that
\[
(q_1 + \cdots + q_k)n_1 \le q_1n_1 + \cdots + q_kn_k \le (q_1 + \cdots + q_k)n_k
\]
for any factorization $n = q_1n_1 + \cdots + q_kn_k$.  

To prove the remaining inequality, first suppose $d = 1$, so that $S = T$.  We consider two cases.  If $n \le n_1n_k$, then 
\[
\frac{1}{n_1}n - \frac{1}{n_k}n = \frac{n(n_k - n_1)}{n_1n_k} \le n_k - n_1,
\]
so every $\ell \in \mathsf L(n)$ satisfies the desired inequality.  Next, suppose $n \ge n_1n_k$, and write $n = n_1n_k + qn_k - r$ with $q, r \in \ZZ_{\ge 0}$ and $0 \le r < n_k$ by the division algorithm.  We have $n_1n_k - r \ge \Frob(S)$ by \cite[Theorem~3.1.1]{diophantinefrob}, so since $\mathsf M(n_1n_k - r) \le n_k$, there exists a factorization of $n$ of length $\ell \le n_k + q$.  As such, 
\[
\ell
\le n_k + q
= n_1 + q + (n_k - n_1)
= \tfrac{1}{n_k}(n + r) + (n_k - n_1)
< \tfrac{1}{n_k}n + 1 + (n_k - n_1),
\]
and since $\ell \in \ZZ$, we have $\ell \le \tfrac{1}{n_k}n + (n_k - n_1)$.  

Lastly, if $d > 1$, then applying the above argument to $S = \tfrac{1}{d}T$, there exists a factorization of $\tfrac{1}{d}n \in S$ of length at most
\[
\tfrac{1}{n_k}n + \tfrac{1}{d}(n_k - n_1) \le \tfrac{1}{n_k}n + (n_k - n_1),
\]
so there must also exist a factorization of this length for $n \in T$.  
\end{proof}

%%%%%%%%%%%%%%%%%%%%%%%%%%%%%%%%%%%%%%%%%%%%%%%%%%%%%%%%%%%%%%%%%%%%%%%%%
\section{Properties of maximum and minimum factorization length}%%%%%%%%%
\label{sec:maxminlen}%%%%%%%%%%%%%%%%%%%%%%%%%%%%%%%%%%%%%%%%%%%%%%%%%%%%
%raggedbottom%%%%%%%%%%%%%%%%%%%%%%%%%%%%%%%%%%%%%%%%%%%%%%%%%%%%%%%%%%%%

The main results of this section are Theorems~\ref{t:jmaxlen} and~\ref{t:jminlen}, wherein we classify the $j$-th maximum and minimum factorization lengths, respectively (Definition~\ref{d:jmaxminlen}) for sufficiently large $n \in S$.  These form the crux of our proof of Theorem~\ref{t:structurethm}, which makes explicit the phenomenon discussed in Example~\ref{e:mainex} and depicted in Figure~\ref{f:maxminlentables}.  Although there is symmetry between the proofs of these two results, we include a proof for each, as there are some subtle differences in the arguments.  

\begin{notation}\label{n:restofpaper}
For the remainder of this paper, unless otherwise stated, fix a numerical semigroup \(S = \langle n_1, n_2, \dots, n_k \rangle\) that is cofinite in $\ZZ_{\ge 0}$.
Write
\[
S_{\M} =  \langle n_2 - n_1, n_3 - n_1, \dots, n_k - n_1 \rangle
\qquad \text{and} \qquad
\Ap_j(S_{\M}; n_1) = \{b_{0j}, b_{1j}, \ldots\}
\]
where each \(b_{ij} \equiv i \bmod n_1\).  Analogously, write 
\[
S_{\MinS} =  \langle n_k - n_1, n_k - n_2, \dots, n_k - n_{k-1} \rangle
\qquad \text{and} \qquad
\Ap_j(S_{\MinS}; n_k) = \{c_{0j}, c_{1j}, \ldots\}
\]
where each \(c_{ij} + i \equiv 0 \bmod n_k\).  Lastly, let $d = \gcd(S_{\M}) = \gcd(S_{\MinS})$, which can be shown to be equal by an elementary number theory argument.  
\end{notation}

% We briefly recall a characterization of $\mathsf M(n)$ and $\mathsf m(n)$ that is known to hold for sufficiently large $n \in S$.  

% \begin{thm}[{\cite[Theorems~4.2 and~4.3]{elastsets}}]\label{t:maxminquasi}
% If $n > (n_{k-1} - 1)n_k$, then 
% \[
% \mathsf M(n + n_1) = \mathsf M(n) + 1
% \qquad \text{and} \qquad
% \mathsf m(n + n_k) = \mathsf m(n) + 1.
% \]
% \end{thm}

\begin{defn}\label{d:jmaxminlen}
Fix $j \in \ZZ_{\ge 1}$, and suppose \(n \in S\) with \(|\LengthSet(n)| \geq j\).  Define \(\M_j(n)\) and \(\MinS_j(n)\) as the \(j\)-th largest and $j$-th smallest factorization lengths of \(n\), respectively. 
In~particular, \(\M_1(n) = \M(n)\) and \(\MinS_1(n) = \MinS(n)\).
\end{defn}

\begin{thm}\label{t:jmaxlen}
If $j \ge 1$, then for all sufficiently large \(n \in S\) with \(n \equiv i \bmod n_1\),
% sufficiently large (that is, there exists some \(N\) such that for all \(n \geq N\)),
\[\M_j(n) = \frac{n - b_{ij}}{n_1}.\]
\end{thm}

\begin{proof}
Fix $n \in S$, and write \(n = pn_1 + i\) for \(p, i \in \ZZ\) with \(0 \le i < n_1\).  
Consider a factorization
\[
n = q_1n_1 + q_2n_2 + \cdots + q_kn_k
\]
of $n$, whose length is $\ell = q_1 + q_2 + \cdots + q_k$. Letting
\begin{equation}\label{eq:factassoc}
b = (p - \ell)n_1 + i = n - \ell n_1 = q_2(n_2 - n_1) + \cdots + q_k(n_k - n_1),
\end{equation}
we see $b \in S_{\M}$ and $b \equiv i \bmod n_1$, so $b = b_{ij}$ for some $j \ge 1$ by Lemma~\ref{l:japeryset}.  Note $i$ and $j$ only depend on $n$ and $\ell$, and not on the specific values of $q_1, \ldots, q_k$.  In particular, we have obtained a map 
\[
f:\mathsf L(n) \to \{b_{ij} : j \ge 1\}
\qquad \text{given by} \qquad
\ell \mapsto n - \ell n_1,
\]
which this associates, to each length $\ell \in \mathsf L(n)$, an element $b_{ij} \in S_{\M}$. 

Now, write $\mathsf L(n) = \{\ell_1 > \ell_2 > \cdots\}$.  We claim that, for each fixed $h \ge 1$, if $n$ is sufficiently large the map $f$ induces a bijection
\begin{equation}\label{eq:maxbijection}
\{\ell_1, \ldots, \ell_h\} \rightsquigarrow \{b_{i1}, \ldots, b_{ih}\}.
\end{equation}
Indeed, fix $j \le h$.  For any factorization
\[
b_{ij} = Q_2(n_2 - n_1) + \cdots + Q_k(n_k - n_1)
\]
of $b_{ij}$ in $S_{\M}$, we have
$$z = b_{ij} + (Q_2 + \cdots + Q_k)n_1 = Q_2n_2 + \cdots + Q_kn_k \in S.$$
As such, choosing $\ell$ so that $n - \ell n_1 = b_{ij}$, if $n \in n_1 \ZZ_{\ge 0} + z$, then 
\[
n = b_{ij} + \ell n_1 = (\ell - Q_2 + \cdots + Q_k)n_1 + Q_2n_2 + \cdots + Q_kn_k
\]
is a factorization of length $\ell \in \mathsf L(n)$, and $f(\ell) = b_{ij}$.  In particular, this proves~\eqref{eq:maxbijection} is a bijection when $n$ is sufficiently large.  As a final step, choosing $\ell = \mathsf M_j(n)$ and solving~\eqref{eq:factassoc} for $\ell$ then yields the desired equality.  
\end{proof}

\begin{thm}\label{t:jminlen}
If $j \ge 1$, then for all sufficiently large \(n \in S\) with \(n \equiv i' \bmod n_k\), 
\[\MinS_j(n) = \frac{n  + c_{i'j}}{n_k}.\]
\end{thm}

\begin{proof}
Fix $n \in S$, and write \(n = pn_k + i'\) for \(p, i' \in \ZZ\) with \(0 \le i' < n_k\).  
If
\[
n = q_1n_1 + q_2n_2 + \cdots + q_kn_k
\]
is a factorization of $n$ with length $\ell = q_1 + q_2 + \cdots + q_k$, then letting
\begin{equation}\label{eq:minbijection}
c = (\ell - p)n_k - i' = \ell n_k - n = q_1(n_k - n_1) + \cdots + q_{k-1}(n_k - n_{k-1}),
\end{equation}
we see $c \in S_{\MinS}$ and $c + i' \equiv 0 \bmod n_k$, so $c = c_{i'j}$ for some $j \ge 1$ by Lemma~\ref{l:japeryset}.  This yields a map 
\[
f:\mathsf L(n) \to \{c_{i'j} : j \ge 1\}
\qquad \text{given by} \qquad
\ell \mapsto \ell n_k - n,
\]
which this associates, to each length $\ell \in \mathsf L(n)$, an element $c_{i'j} \in S_{\MinS}$. 
Now, writing $\mathsf L(n) = \{\ell_1 < \ell_2 < \cdots\}$, we can show by a similar argument to the proof of Theorem~\ref{t:jmaxlen} that for each fixed $h \ge 1$, if $n$ is sufficiently large the map $f$ induces a bijection
\[
\{\ell_1, \ldots, \ell_h\} \rightsquigarrow \{c_{i'1}, \ldots, c_{i'h}\}.
\]
Solving~\eqref{eq:minbijection} for $\ell = \mathsf m_j(n)$ completes the proof.  
\end{proof}

% \begin{example}\label{e:jmaxlen}
% % Let \(S = \langle 5, 9, 12 \rangle\), for which
% % \(S_{\M} = \langle 4, 7 \rangle.\)
% % One can compute \(\M_1(100) = 20\) and \(b_{01} = 0\), which agrees with Theorem~\ref{t:jmaxlen}.  
% % In fact, we note that 
% Each factorization length in Figure~\ref{f:maxlenset} can be calculated with Theorem~\ref{t:jmaxlen} using the value \(b_{ij}\) obtained from the corresponding entry in Figure~\ref{f:maxlenapery}.  In~fact, Theorem~\ref{t:jmaxlen} proves the relative positions of the missing entries in Figure~\ref{f:maxlenapery} coincide with those in Figure~\ref{f:maxlenset}, as was aluded to in Example~\ref{e:mainex}.  
% An~analogous correspondence is depicted in Figures~\ref{f:minlenapery} and~\ref{f:minlenset} after reflection; this reflection is the reason that subtraction appears in the numerator in Theorem~\ref{t:jmaxlen} while addition appears in the numerator in Theorem~\ref{t:jminlen}.
% \end{example}

\begin{remark}\label{r:maxminquasi}
It was proven in \cite[Theorems~4.2 and~4.3]{elastsets} that 
\[
\mathsf M(n + n_1) = \mathsf M(n) + 1
\qquad \text{and} \qquad
\mathsf m(n + n_k) = \mathsf m(n) + 1
\]
for sufficiently large $n \in S$.  Corollary~\ref{c:jmaxminquasi} (below) is a generalization of this result.  

Another way to state this result is that there exist $n_1$- and $n_k$-periodic functions $f_S(n)$ and $g_S(n)$, respectively, such that for all sufficiently large $n \in S$,
\[
\mathsf M(n) = \tfrac{1}{n_1}n + f_S(n)
\qquad \text{and} \qquad
\mathsf m(n) = \tfrac{1}{n_k}n + g_S(n).
\]
The question was posed in~\cite[Project~3]{beyondcoins} to characterize the functions $f_S$ and $g_S$ in terms of the generators of $S$.  Theorems~\ref{t:jmaxlen} and~\ref{t:jminlen} answer this question, expressing $f$ and $g$ in terms of the elements of $\Ap(S_{\M}; n_1)$ and $\Ap(S_{\MinS}; n_k)$, respectively.  It was also asked in~\cite{beyondcoins} whether it is possible $f_S = f_{S'}$ and $g_S = g_{S'}$ for distinct numerical semigroups $S$ and $S'$; in addition to identfying when this occurs in terms of Ap\'ery sets, our results provide a rubric for constructing examples.  For instance, consider
\[
S = \<10,16,44,49,51\>
\qquad \text{and} \qquad
S' = \<10,16,38,44,49,51\>.
\]
It is not hard to check $S_{\MinS} = S_{\MinS}' = \<2,7\>$, so $\Ap(S_{\MinS}; 51) = \Ap(S_{\MinS}; 51)$ and thus $g_S = g_{S'}$.  However, $S_{\M}' \setminus S_{\M} = \{28,56,67\}$, even though
\[
\Ap(S_{\M};10) = \Ap(S_{\M}'; 10) = \{0, 41, 12, 53, 24, 45, 6, 47, 18, 39\}
\]
and thus $f_S = f_{S'}$.  
\end{remark}

\begin{cor}\label{c:jmaxminquasi}
Fix $j \ge 1$.  For all sufficiently large $n \in S$, we have
\[
\mathsf M_j(n + n_1) = \mathsf M_j(n) + 1
\qquad \text{and} \qquad
\mathsf m_j(n + n_k) = \mathsf m_j(n) + 1.
\]
\end{cor}

\begin{proof}
Apply Theorems~\ref{t:jmaxlen} and~\ref{t:jminlen}.  
\end{proof}

% \begin{cor}\label{c:jmaxlen}
% Every \(n \in S\) satisfies \(\M_1(n) \le \tfrac{1}{n_1}(n - b_{i1})\), and strict equality holds for only finitely many \(n \in S\). 
% \end{cor}

% \begin{proof}
% The second claim follows immediately from Theorem~\ref{t:jmaxlen} once we prove the first claim.  
% By way of contradiction, suppose \(\M_1(n) > \tfrac{1}{n_1}(n - b_{i1})\) for some \(n \in S\).  There exists \(s \gg 0\) such that \(\M_1(s) = \tfrac{1}{n_1}(s - b_{i1})\) by Theorem~\ref{t:jmaxlen}.  Writing \(s = n + pn_1\) for some \(p \in \ZZ\), then we have 
% \begin{align*}
%     \M_1(s) &= \frac{s - b_{i1}}{n_1}
%     = \frac{(n + pn_1) - b_{i1}}{n_1}
%     = \frac{n - b_{i1}}{n_1} + p.
% \end{align*}
% We then also have that
% \begin{align*}
%     \frac{n - b_{i1}}{n_1} + p = \M_1(n + pn_1) &\geq \M_1(n) + p > \frac{n - b_{i1}}{n_1} + p,
% \end{align*}
% which is a contradiction. 
% % Thus there is no \(n \in S\) such that \(\M_1(n) > \frac{n - b_{i1}}{n_1}\).
% \end{proof}

%%%%%%%%%%%%%%%%%%%%%%%%%%%%%%%%%%%%%%%%%%%%%%%%%%%%%%%%%%%%%%%%%%%%%%%%%
\section{The refined structure theorem for numerical semigroups}%%%%%%%%%
\label{sec:structurethm}%%%%%%%%%%%%%%%%%%%%%%%%%%%%%%%%%%%%%%%%%%%%%%%%%
%raggedbottom%%%%%%%%%%%%%%%%%%%%%%%%%%%%%%%%%%%%%%%%%%%%%%%%%%%%%%%%%%%%

In this section, we prove our main result:\ a refinement of the structure theorem for sets of length for numerical semigroups (Theorem~\ref{t:structurethm}).  We also give an explicit bound on when the structure theorem holds (Theorem~\ref{t:bound}) and discuss the ramifications of this bound (Remark~\ref{r:deltaset}).  

\begin{notation}\label{n:maxminmissinglens}
For each $i \in \{0, 1, \ldots, n_1-1\}$, let 
\[
A_i = \{r \in \ZZ_{\ge 1} : b_{i1} + r d n_1 \notin S_{\M}\}
\]
and for each $i' \in \{0, 1, \ldots, n_k-1\}$, let
\[
A'_{i'} = \{r' \in \ZZ_{\ge 1} : c_{i'1} + r' d n_k \notin S_{\MinS}\}.
\]
% be the set of integers corresponding the \(i\)-th equivalence class modulo \(n_1\) correlated with the gaps of \(S_{\M}\) that appear beyond \(b_{i1}\).
% Additionally, let
% \[
% t = \bigg\lceil \frac{\Frob(S_{\M}) + 1}{dn_1} \bigg\rceil - 1
% \quad \text{and} \quad
% t' = \bigg\lceil \frac{\Frob(S_{\MinS}) + 1}{dn_k} \bigg\rceil - 1
% \]
\end{notation}

\begin{thm}\label{t:structurethm}
For all sufficiently large \(n \in S\) with \(n \equiv i \bmod n_1\) and \(n \equiv i' \bmod n_k\),
\[\LengthSet(n) = \{\MinS(n), \MinS(n) + d, \dots, \M(n) - d, \M(n)\} \setminus \big( (dA'_{i'} + \MinS(n)) \cup (-dA_i + \M(n)) \big)\]
\end{thm}

\begin{proof}
By the structure theorem for sets of length and \cite[Proposition~2.9]{delta}, there exist $t, t' \in \ZZ_{\ge 1}$ such that for all sufficiently large $n \in S$, 
% \[
% \{\MinS(n) + t'd, \MinS(n) + (t'+1)d, \ldots, \M(n) - td\} \subseteq \mathsf L(n)
% \]
\[
(\mathsf m(n) + d\ZZ) \cap [\MinS(n) + t'd, \M(n) - td] \subseteq \mathsf L(n)
\]
Fix $n \in S$ with \(n \equiv i \bmod n_1\) and \(n \equiv i' \bmod n_k\) large enough that (i)~the above holds, (ii)~Theorem~\ref{t:jmaxlen} holds for $j \le t$, and (iii)~Theorem~\ref{t:jminlen} holds for $j \le t'$.  
% Fix 
% \[
% \ell \in \{\MinS(n), \MinS(n) + d, \ldots, \M(n) - d, \M(n)\}.
% \]

First, suppose $\ell = \M(n) - rd$ for some $r \le t$, and let $b = b_{i1} + rdn_1$.  If $r \notin A_i$, then $b \in S_\M$, meaning $b = b_{ij}$ for some $j$, and thus 
\[
\ell
= \M(n) - rd
= \frac{n - b_{i1} - rdn_1}{n_1}
= \frac{n - b_{ij}}{n_1}
= M_j(n)
\in \mathsf L(n)
\]
by Theorem~\ref{t:jmaxlen}.  Conversely, if $\ell \in \mathsf L(n)$, then since $r \le t$, Theorem~\ref{t:jmaxlen} implies 
\[
\ell = \M_j(n) = \frac{n - b_{ij}}{n_1}
\]
for some $j$.  Rearranging, we find 
\[
b_{ij}
= n - \ell n_1
= n - (\M(n) - rd)n_1
= n - (n - b_{i1}) + rdn_1
= b_{i1} + rdn_1
= b
\]
which means $b \in S_\M$
and thus $r \notin A_i$.  

Now, by an analogous argument, if $\ell = \MinS(n) + dr'$ for some $r' \le t'$, then Theorem~\ref{t:jminlen} implies $\ell \in \mathsf L(n)$ if and only if $r' \notin A_{i'}'$.  This completes the proof.  
% \[
% (dA_i + \M(n)) \cap \mathsf L(n) = \emptyset
% \qquad \text{and} \qquad
% (dA'_{i'} + \MinS(n)) \cap \mathsf L(n) = \emptyset,
% \]
% respectively, for all sufficiently large $n$.  
\end{proof}

\begin{remark}
It was shown in~\cite[Corollary~5.5]{factorhilbert} that 
\[
|\LengthSet(n + n_1n_k)| = |\LengthSet(n)| + \tfrac{1}{d}(n_k - n_1)
\]
for sufficient large \(n \in S\).  This also an immediate consequence of Theorem~\ref{t:structurethm}.  
% , since \(n \bmod n_1 = n  + \lcm(n_1, n_k) \bmod n_1\) and \(n \bmod n_k = n  + \lcm(n_1, n_k) \bmod n_k\).
\end{remark}

In the remainder of this section, we identify an explicit bound on the ``sufficiently large $n \in S$'' in the statement of Theorem~\ref{t:structurethm}.  First, we obtain the constants $t$ and $t'$ in the (unrefined) structure theorem for sets of length.  

\begin{prop}\label{p:tbounds}
For each $i$ and $i'$, we have $A_i \subseteq A_0$ and $A_{i'}' \subseteq A_0'$.  In~particular,  
\[
t = \max(A_0)
% = \max(\ZZ_{\ge 0} \setminus \tfrac{1}{dn_1}(S_{\M} \cap dn_1\ZZ))
\qquad \text{and} \qquad
t' = \max(A_0'),
% = \max(\ZZ_{\ge 0} \setminus \tfrac{1}{dn_k}(S_{\MinS} \cap dn_k\ZZ)),
\]
% \[
% t = \bigg\lfloor \frac{\Frob(S_{\M})}{dn_1} \bigg\rfloor
% \quad \text{and} \quad
% t' = \bigg\lfloor \frac{\Frob(S_{\MinS})}{dn_k} \bigg\rfloor.
% \]
are the minimal values so that $A_i \subseteq [1, t]$ and $A_{i'}' \subseteq [1, t']$ for all $i$ and $i'$, respectively.  
% The value of \(t\) 
% is the smallest integer value for which
% \[\min(b_{i1}) + (t + 1) \cdot d \cdot n_1 > \Frob(S_{\M}).\]
% Then \(t\) 
% is minimal such that for all \(i\) we have \(A_i \subseteq [1, t]\).
\end{prop}

\begin{proof}
% Notice that \(t\) is the smallest integer value for which
% \[
% \min(b_{i1}) + d n_1 (t + 1) > \Frob(S_{\M}).
% \]
% If $q \in A_i$, then $g = b_{i1} + qdn_1 \notin S_{\M}$.  Since $b_{i1} \in S_{\M}$, $d \mid b_{i1}$ and thus $d \mid g$, so~$g \le \Frob(S_{\M})$.  As such, since $q \in \ZZ$ and
% \[
% q = \frac{g - b_{i1}}{dn_1} \le \frac{\Frob(S_{\M}) - b_{i1}}{dn_1} \le \frac{\Frob(S_{\M})}{dn_1},
% \]
% we have $q \le t$.  An analogous argument proves the remaining claims.  
If $r \in \ZZ_{\ge 1} \setminus A_0$, then 
\[
b = b_{01} + dn_1r = dn_1r \in S_{\M}.
\]
This means, for any $i$, we have 
\[
b + b_{i1} = b_{i1} + dn_1r \in S_{\M},
\]
so $r \in \ZZ_{\ge 1} \setminus A_i$.  This proves $A_i \subseteq A_0$.  
An analogous argument proves each $A_{i'}' \subseteq A_0'$, and the remaining claims follow from
the fact that 
\[
t = \max(A_0)
\qquad \text{and} \qquad
t' = \max(A_0'),
\]
and from applying Theorems~\ref{t:jmaxlen} and~\ref{t:jminlen}.  
\end{proof}

\begin{remark}\label{r:classification}
In addition to yielding upper bounds
\[
t \le \tfrac{1}{dn_1}\Frob(S_{\M})
\qquad \text{and} \qquad
t' \le \tfrac{1}{dn_k}\Frob(S_{\MinS})
\]
on $t$ and $t'$ in the structure theorem, Proposition~\ref{p:tbounds} has implications on questions concerning of which combinations of ``missing'' lengths can occur, which have been considered for other families of semigroups~\cite{subdeltas,deltarealizationnumerical,deltarsetealization,setdistanceskrull}.  
Letting 
\[
\mathcal A_S = \{A_i : 0 \le i \le n_1 - 1\}
\qquad \text{and} \qquad
\mathcal A_S' = \{A_{i'}' : 0 \le i' \le n_k - 1\},
\]
Proposition~\ref{p:tbounds} implies $\bigcup \mathcal A_S \in \mathcal A_S$ and $\bigcup \mathcal A_S' \in \mathcal A_S'$, the first known restrictions on $\mathcal A_S$ and $\mathcal A_S'$ for numerical semigroups.  
Additionally, under the mild assumption $\gcd(n_1, n_k) = 1$, there are infinitely many $n \in S$ for which
\[\LengthSet(n) = \{\MinS(n), \MinS(n) + d, \dots, \M(n) - d, \M(n)\} \setminus \big( (dA' + \MinS(n)) \cup (-dA + \M(n)) \big)\]
for each pair $A \in \mathcal A_S$ and $A' \in \mathcal A_S'$.  
In particular, in order to classify the possible combinations of ``missing'' lengths from the ``top'' and ``bottom'' of the length sets of large $n \in S$, it suffices to classify $\mathcal A_S$ and $\mathcal A_S'$ independently.  
% As such, in order to obtain a realization theorem of the form proven in~\cite{realizthm}, it suffices to locate $S$ for which every subset of $A_0 = [1,t]$ and $A_0 = [1, t']$ occurs as some $A_i$ and $A_{i'}'$, respectively.  
% It is proven in~\cite{realizthm} that within the family of Krull monoids with finite class group, there exist monoids for which every 
\end{remark}

In view of Remark~\ref{r:classification}, we state the following question, posed by Geroldinger in private communication with the second author and answered in the affirmative in~\cite{realizthm} for the family of Krull monoids with finite class group.  

\begin{question}\label{q:geroldinger}
Given $d, t, t' \in \ZZ_{\ge 1}$, does there exist a numerical semigroup $S$ such that $\mathcal A_S$ and $\mathcal A_S'$ equal the power sets of $[1,t]$ and $[1,t']$, respectively, and $d = \gcd(S_{\M})$?  
\end{question}

% Let \(f : S_{\M} \to S\) and \(g : S_{\MinS} \to S\) be given by
% \[
% f(b) = \M_{S_{\M}}(b)n_1 + b
% \qquad \text{and} \qquad
% g(c) = \M_{S_{\MinS}}(c)n_k - c,
% \]
% where $\M_{S_{\M}}(b)$ and $\M_{S_{\MinS}}(c)$ denote the maximum factorization lengths of $b$ and $c$ in $S_{\M}$ and $S_{\MinS}$, respectively.  
% We have $f(b) \in S$ for each $b \in S_{\M}$ since $(f(b), \M_{S_{\M}}(b))$ corresponds to $b$ under the correspondence establisehd in the proof of Theorem~\ref{t:jmaxlen}.  Similarly, $f(c) \in S$ for each $c \in S_{\MinS}$ by the proof of Theorem~\ref{t:jminlen}.  In particular, \(f(b)\)~has a factorization of length \(\M_{S_{\M}}(b)\) and \(g(c)\) has a factorization  of length \(\M_{S_{\MinS}}(c)\).

\begin{thm}\label{t:bound}
Theorem~\ref{t:structurethm} holds for all \(n \ge n_k^2 - n_1^2\).  
% \[\LengthSet(n) = \{\MinS(n), \MinS(n) + d, \dots, \M(n) - d, \M(n)\} \setminus \big( (-dA_i + \M(n)) \cup (dA'_{i'} + \MinS(n)) \big)\]
% where \(i = n \bmod n_1\) and \(i' = n \bmod n_k\), with \(d\), \(t\), \(t'\), \(A_i\) and \(A'_{i'}\) as defined in Theorem~\ref{t:structurethm}. 
\end{thm}

\begin{proof}
Suppose $n \ge n_k^2 - n_1^2$.  
Fix $\ell \in [\tfrac{1}{n_k}n, \tfrac{1}{n_1}n] \cap \ZZ$.  First, suppose 
\[
\ell \ge \frac{n}{(n_k + n_1)/2} = \frac{2n}{n_k + n_1},
\]
and let $b = n - \ell n_1$.  If $b \notin S_{\M}$, then $\ell + q \notin \mathsf L(n + q n_1)$ for all sufficiently large $q$ by Theorem~\ref{t:structurethm}, so $\ell \notin \mathsf L(n)$.  If $b \in S_{\M}$, then applying Lemma~\ref{l:minlenupperbound} to $S_{\M}$, there exists a factorization of $b \in S_{\M}$ of length at most $\ell$ since 
% \begin{align*}
% \ell
% &\ge \frac{\ell n_k + ()}{n_k - n_1}
% = \frac{\ell n_k - (n_k^2 - n_1^2)}{n_k - n_1} + (n_k - n_2)
% \ge \frac{\ell n_k - n}{n_k - n_1} + (n_k - n_2)
% \\
% &= \frac{n - (2n - \ell n_k)}{n_k - n_1} + (n_k - n_2)
% % \ge \frac{n - \ell n_1}{n_k - n_1} + (n_k - n_2)
% \ge \frac{n - \ell n_1}{n_k - n_1} + (n_k - n_2)
% = \frac{b}{n_k - n_1} + (n_k - n_2).
% \end{align*}
% \begin{align*}
% \ell
% &\ge \ell + \frac{\ell n_1}{n_k - n_1} - n_1 - n_2
% = \ell + \frac{\ell n_1 - (n_k^2 - n_1^2)}{n_k - n_1} + (n_k - n_2)
% \\
% &\ge \ell + \frac{\ell n_1 - n}{n_k - n_1} + (n_k - n_2)
% = \ell + \frac{n - (2n - \ell n_1)}{n_k - n_1} + (n_k - n_2)
% \\
% &\ge \ell + \frac{n - \ell n_k}{n_k - n_1} + (n_k - n_2)
% = \frac{n - \ell n_1}{n_k - n_1} + (n_k - n_2)
% = \frac{b}{n_k - n_1} + (n_k - n_2).
% \end{align*}
\begin{align*}
\ell (n_k - n_1)
&\ge \frac{2n n_k}{n_k + n_1} - \ell n_1
% = \frac{n(n_k + n_1 + n_k - n_1)}{n_k + n_1} - \ell n_1
= n - \ell n_1 + \frac{n(n_k - n_1)}{n_k + n_1}
\\
&\ge n - \ell n_1 + (n_k - n_1)^2
\ge n - \ell n_1 + (n_k - n_1)(n_k - n_2)
\end{align*}
implies
\[
\ell
\ge \frac{n - \ell n_1}{n_k - n_1} + (n_k - n_2)
= \frac{b}{n_k - n_1} + \big( (n_k - n_1) - (n_2 - n_1) \big).
\]
As such, we apply the correspondence in the proof of Theorem~\ref{t:jmaxlen}:\ if 
\[
b = q_2(n_2 - n_1) + \cdots + q_k(n_k - n_1)
\]
is a factorization of $b \in S_{\M}$ of length at most $\ell$, then
\[
n = b + \ell n_1 = (\ell - q_2 - \cdots - q_k)n_1 + q_2n_2 + \cdots + q_kn_k
\]
is a factorization of $n \in S$ of length exactly $\ell$, so $\ell \in \mathsf L(n)$.  

Next, suppose
\[
\ell \le \frac{2n}{n_k + n_1},
\]
and let $c = \ell n_k - n$.  If $c \notin S_{\MinS}$, then $\ell + q \notin \mathsf L(n + qn_k)$ for all sufficiently large $q$ by Theorem~\ref{t:structurethm}, so $\ell \notin \mathsf L(n)$.  If $c \in S_{\MinS}$, then applying Lemma~\ref{l:minlenupperbound} to $S_{\MinS}$, there exists a factorization of $c \in S_{\MinS}$ of length at most $\ell$ since
\begin{align*}
\ell (n_k - n_1)
&\ge \ell n_k - \frac{2nn_1}{n_1 + n_k}
= \ell n_k - n + \frac{n(n_k - n_1)}{n_k + n_1}
\\
&\ge \ell n_k - n + (n_k - n_1)^2
\ge \ell n_k - n + (n_k - n_1)(n_k - n_{k-1})
\end{align*}
implies
\[
\ell
\ge \frac{\ell n_k - n}{n_k - n_1} + (n_k - n_{k-1})
= \frac{c}{n_k - n_1} + (n_k - n_{k-1})
\]
Thus, as before, $\ell \in \mathsf L(n)$ by the correspondence in the proof of Theorem~\ref{t:jminlen}.  
\end{proof}

\begin{remark}\label{r:topbottomoverlap}
If $n \ge n_k^2 - n_1^2$, the ``top'' and ``bottom'' of the length set (as described in Proposition~\ref{p:tbounds}) do not overlap.  Indeed, by Theorem~\ref{t:structurethm}, if $\ell \in [\tfrac{1}{n_k}n, \tfrac{1}{n_1}n] \cap \ZZ$ with $\ell \notin \mathsf L(n)$, then
\[
\ell \in
[\mathsf m(n), \mathsf m(n) + \tfrac{1}{n_k}\Frob(S_{\MinS})] 
\cup 
[\mathsf M(n) - \tfrac{1}{n_1}\Frob(S_{\M}), \mathsf M(n)],
\]
from which we obtain
\begin{align*}
\frac{n}{n_1} - \frac{n}{n_k}
&= n\frac{n_k - n_1}{n_1n_k}
\ge (n_2n_k - n_1n_{k-1})\frac{n_k - n_1}{n_1n_k}
= \big( n_k(n_2 - n_1) + n_1(n_k - n_{k-1}) \big)\frac{n_k - n_1}{n_1n_k}
\\
&= \tfrac{1}{n_1}(n_k - n_1)(n_2 - n_1) + \tfrac{1}{n_k}(n_k - n_1)(n_k - n_{k-1})
\ge \tfrac{1}{n_1}\Frob(S_{\M}) + \tfrac{1}{n_k}\Frob(S_{\MinS}),
\end{align*}
where the final inequality follows from \cite[Theorem~3.1.1]{diophantinefrob}. 
\end{remark}

\begin{remark}\label{r:deltaset}
Given $n \in S$ and writing $\mathsf L(n) = \{\ell_1 < \cdots < \ell_r\}$, the \emph{delta set} of $n$ is 
\[
\Delta(n) = \{\ell_i - \ell_{i-1} : i \le r\}.
\]
It is known that $\Delta(n) = \Delta(n + \lcm(n_1,n_k))$ for all $n \ge 2kn_2n_k^2$~\cite{deltaperiodic}, and some effort has been made to refine this bound~\cite{compasympdelta} and to compute delta sets explicitly~\cite{dynamicalg,affineinvariantcomp}.  
Theorem~\ref{t:bound}, in addition to providing an explicit bound for Corollary~\ref{c:jmaxminquasi}, identifies a bound on the start of periodicity for the delta set.  Our bound appears to be better on average than the one obtained in~\cite{compasympdelta} (in a sample of 10000 randomly selected numerical semigroups with $k \le 10$ and $n_k \le 10000$, our bound was better in roughly 75\% of cases), as well as more concise (the one in~\cite{compasympdelta} takes the better part of a page to write down).  
\end{remark}

\begin{remark}\label{r:boundspecialcases}
If \(n_2 - n_1 = d\), then \(S_{\M}\) has no gaps, and so \(A_i= \varnothing\) for all \(i\).  Analogously, if \(n_k - n_{k-1} = d\), then \(S_{\MinS}\) has no gaps and thus \(A'_{i'}= \varnothing\) for all \(i'\).  In particular, if both of these are satisfied, then, for \(n\) sufficiently large, every length set is an arithmetic sequence with step size \(d\).  Note that the ``sufficiently large $n$'' is necessary even in this special case.  For example, if \(n = 26 \in S = \langle 5, 6, 13, 14 \rangle\), then \(\LengthSet(n) = \{2, 5\}\).
% since 
% \[
% 26 = 2 \cdot 13 = 4 \cdot 5 + 1 \cdot 6.
% \]
This is an improvement on \cite[Corollary~3.6]{gensomitted}, which relates the length sets of element of a numerical semigroup generated by an arithmetic sequence to one in which ``middle generators'' are omitted.  
\end{remark}

%%%%%%%%%%%%%%%%%%%%%%%%%%%%%%%%%%%%%%%%%%%%%%%%%%%%%%%%%%%%%%%%%%%%%%%%%
\section*{Acknowledgements}%%%%%%%%%%%%%%%%%%%%%%%%%%%%%%%%%%%%%%%%%%%%%%
%raggedbottom%%%%%%%%%%%%%%%%%%%%%%%%%%%%%%%%%%%%%%%%%%%%%%%%%%%%%%%%%%%%

The authors would like to thank Scott Chapman, Alfred Geroldinger, and Vadim Ponomarenko for their feedback and helpful conversations.

%%%%%%%%%%%%%%%%%%%%%%%%%%%%%%%%%%%%%%%%%%%%%%%%%%%%%%%%%%%%%%%%%%%%%%%%%
%%%%%%%%%%%%%%%%%%%%%%%%%%%%%%%%%%%%%%%%%%%%%%%%%%%%
%%%%%%%%%%%%%%%%%%%%%%%%%%%%%%%%%%%%%%%%%%%%%%%%%%%%%%%%%%%%%%%%%%%%%%%%%

%%%%%%%%%%%%%%%%%%%%%%%%%%%%%%%%%%%%%%%%%%%%%%%%%%%%%%%%%%%%%%%%%%%%%%%%%
\end{document}